\documentclass[10pt,a4paper,reqno]{amsart}
\usepackage{amsmath,amssymb,amsthm,amsfonts}
\usepackage[utf8]{inputenc}
\usepackage{geometry}
\usepackage{enumitem}
\usepackage[english]{babel}
\usepackage{url}
\usepackage{graphicx}
\usepackage{color}

\newcommand{\cut}{\textsf{cut}}
\newcommand{\Irr}{\operatorname{Irr}}
\newcommand{\Aut}{\operatorname{Aut}}
\newcommand{\Syl}{\operatorname{Syl}}
\newcommand{\ccl}{\operatorname{ccl}}
\newcommand{\GL}{\operatorname{GL}}
\newcommand{\PSL}{\operatorname{PSL}}

\newcommand{\Z}{\mathcal{Z}}
\newcommand{\U}{\mathcal{U}}
\newcommand{\C}{\textup{C}}
\newcommand{\N}{\textup{N}}
\newcommand{\ZZ}{\mathbb{Z}}
\newcommand{\QQ}{\mathbb{Q}}

\makeatletter
\newcommand{\slunlhd}{%
  \mathrel{\mathpalette\sl@unlhd\relax}%
}

\newcommand{\sl@unlhd}[2]{%
  \sbox\z@{$#1\lhd$}%
  \sbox\tw@{$#1\leqslant$}%
  \dimen@=\ht\tw@
  \advance\dimen@-\ht\z@
  \ifx#1\displaystyle
    \advance\dimen@ .2pt
  \else
    \ifx#1\textstyle
      \advance\dimen@ .2pt
    \fi
  \fi
  \ooalign{\raisebox{\dimen@}{$\m@th#1\lhd$}\cr$\m@th#1\leqslant$\cr}%
}
\makeatother

\newtheorem{theorem}{Theorem}[section]
\newtheorem{proposition}[theorem]{Proposition}
\newtheorem{lemma}[theorem]{Lemma}
\newtheorem{corollary}[theorem]{Corollary}

\newtheorem{maintheorem}{Theorem}

\theoremstyle{definition}

\newtheorem{question}[theorem]{Question}

\newtheorem{remark}[theorem]{Remark}
\newtheorem{example}[theorem]{Example}

\definecolor{LinkColor}{rgb}{0,0,0} %black
\usepackage[colorlinks=true,linkcolor=LinkColor,citecolor=LinkColor,urlcolor= LinkColor, naturalnames, hyperindex, pdfstartview=FitH, bookmarksnumbered, plainpages]{hyperref} 

\renewcommand{\O}{\textup{O}}

\title[Global and local properties of finite groups with finite $\Z(\U(\ZZ G))$]{Global and local properties of finite groups with only finitely many central units in their integral group ring}
\author[A.~B\"achle]{Andreas B\"achle}
\address{(Andreas B\"achle)}
\email{\href{mailto:ABaechle@gmx.net}{ABaechle@gmx.net}}

\author[M. Caicedo]{Mauricio Caicedo}
\address{(Mauricio Caicedo) Vakgroep Wiskunde, Vrije Universiteit Brussel, Pleinlaan 2, 1050 Brussels, Belgium}
\email{\href{mailto:mcaicedo@vub.be}{mcaicedo@vub.be}}

\author[E. Jespers]{Eric Jespers}
\address{(Eric Jespers) Vakgroep Wiskunde, Vrije Universiteit Brussel, Pleinlaan 2, 1050 Brussels, Belgium}
\email{\href{mailto:eric.jespers@vub.be}{eric.jespers@vub.be}}

\author[S. Maheshwary]{Sugandha Maheshwary}
\address{(Sugandha Maheshwary) Indian Institute of Science Education and Research, Mohali, Sector 81, Mohali (Punjab)-140306, India}
\email{\href{mailto:sugandha@iisermohali.ac.in}{sugandha@iisermohali.ac.in}}

\thanks{The work of the first and the second author was supported by postdoctoral fellowships of the FWO (Research Foundation Flanders). 
The third author is supported in part by Onderzoeksraad of Vrije Universiteit Brussel and FWO (Research Foundation Flanders). 
The research of the fourth author is supported by Department of Science and Technology (DST), India (INSPIRE/04/2017/000897).}

\subjclass[2010] {16S34, 16U60, 20E25} 
\keywords{integral group rings, central units, \cut\ groups, \cut\ property, inverse semi-rational, rational groups, Sylow subgroups}

\begin{document}

\maketitle

\begin{abstract}
The aim of this article is to explore global and local properties of finite groups whose
	integral group rings have only trivial central units, so-called cut groups. For such a group we study actions of Galois groups on its character table and show that the natural actions on the rows and columns are essentially the same, in particular the number of rational-valued irreducible characters coincides with the number of rational-valued conjugacy classes. Further, we prove a natural criterion for nilpotent groups of class 2 to be cut and give a complete list of simple cut groups. Also, the impact of the cut property on Sylow 3-subgroups is discussed. We also collect substantial data on groups which indicates that the class of cut groups is surprisingly large. Several open problems are included.
\end{abstract}

\section{Introduction} 

Let $G$ be a finite group and let $\mathcal U(\mathbb ZG)$ denote the group of units of its integral group ring $\mathbb{Z}G$. The most prominent elements of $\mathcal U(\mathbb ZG)$ are surely $\pm G$, the \emph{trivial units}. In the case these are all the units, this gives tight control on the group $G$ and all the groups with this property were explicitly described by G.~Higman \cite{Hig40}. If the condition is only put on the central elements, that is, $\mathcal{Z}(\mathcal U(\mathbb ZG))$, the center of the units of $\mathbb{Z}G$, only consists of the ``obvious'' elements, namely $\pm \mathcal{Z}(G)$, then the situation is vastly less restrictive and these groups are far from being completely understood. We will however show that this condition is restrictive enough to reveal many of their interesting properties. These groups $G$ with $\mathcal{Z}(\mathcal U(\mathbb ZG))= \pm \mathcal Z(G)$, i.e., all central units of $\mathbb{Z}G$ are trivial, are called \emph{\cut\ groups}, a name coined in \cite{BMP17}. 
The question of classifying \cut\ groups is also included in the collection of major open group ring problems in S.K.~Sehgal's path-breaking book \cite[Problem~26]{Seh93}.  

The study of \cut\ group dates back at least to the 1970s and some cornerstones were a chapter in a book of A.A.~Bovdi \cite{Bov87} and an article by J.~Ritter and Sehgal \cite{RS90}. There the condition for being a \cut\ group was reformulated in a group theoretic condition on $G$, also termed the RS-property (see \cite{BMP18}). An immediate consequence is that rational groups (i.e.\ groups whose character tables only contain rational entries), e.g.\ symmetric groups $S_n$ and Weyl groups of complex Lie algebras, are \cut\ groups. The class of rational groups is of vital interest in the representation theory of finite groups and gives a nice connection to this area. In \cite{CD10}, D.~Chillag and S.~Dolfi defined and studied so-called inverse semi-rational groups from a group theoretic perspective; it turned out that these groups are exactly the same groups  as \cut\ groups (cf.\ Proposition~\ref{prop:equiv_cut}). The \cut\ groups also play a role in other domains and they have been taken up by several authors under different names which exhibits a highly interesting interplay between  group theory, representation theory, algebraic number theory and even $K$-theory (see e.g.\ \cite[Section~3]{MP18} for a complete survey).

One of the original reasons for studying \cut\ groups is that $\mathcal U(\mathbb ZG)$ has a subgroup of finite index that is generated by the central units and the units of reduced norm one for all finite groups $G$. For many $G$ the latter group is generated by very specific unipotent units (called bicyclic units). These, together with the Bass units (a natural generalization of cyclotomic units) ``determine'' finitely many generators of the center. It is thus a natural question to determine when the central units can be avoided to determine finitely many generators of a large subgroup of $\mathcal U(\mathbb ZG)$, i.e., characterize the \cut\ groups. The problem of finding such a finite set of generators of large subgroups of $\mathcal U(\mathbb ZG)$ is a vibrant topic since several decades; for details and many references we refer to \cite[Chapter 11]{JdR15}.

Due to the fact that the rank of the center of $\mathcal{U}(\mathbb{Z}G)$ bounds the rank of its abelianization, and the latter is an important obstruction for certain fixed point properties (like Serre's property (FA)), interest in \cut\ groups arose recently from yet another perspective. This connection was used in \cite{BJJKT18_1,BJJKT18_2} to describe when $\mathcal{U}(\mathbb{Z}G)$ has certain fixed point properties and when it can be decomposed as a non-trivial amalgamated product up to commensurability. The relation between the abelianization and the center of $\mathcal{U}(\mathbb{Z}G)$ was further examined in \cite{BMM20}.

Many interesting properties of \cut\ groups are known. For instance, it follows from Higman's results \cite{Hig40} that an abelian group $G$ is \cut\ if and only if its exponent divides 4 or 6. The work of Ritter and Sehgal \cite{RS90} yields that the \cut\  property is quotient closed. It is observed by G.K.~Bakshi, S.~Maheshwary and I.B.S~Passi \cite{BMP17} that the center of a \cut\ group is again a \cut\ group, yet the \cut\ property is not preserved under taking direct products. They also proved that every \cut\ group has order divisible by $2$ or $3$. Furthermore, it is known that primes dividing the order of \cut\ groups in certain classes are strongly restricted \cite{CD10,BMP17,Mah18,Bac17}, e.g.\ A.~B{\"a}chle proved that only the primes $2$, $3$, $5$ and $7$ divide the order of a solvable \cut\ group \cite{Bac17}. Also explicit descriptions of all \cut\ groups in various classes of groups are obtained in these articles.

In the present work, we continue to examine the class of finite \cut\ groups. After setting up the necessary background in Section~\ref{sect:preliminaries}, we investigate  varied global and local properties of \cut\ groups. In Section~\ref{sect:rationality}, we observe the impact of the \cut\ property on character tables and show that the natural actions of a Galois group on the rows and the columns of the character table of a \cut\ group are basically the same. Precisely, in Theorem~\ref{theo:perm_iso_Galoi_actions} we prove the following:  

\begin{maintheorem}\label{prop:rational_classes_vs_characters} Let $G$ be a \cut\ group of exponent dividing $n$ and let $\zeta$ be a primitive $n^{th}$ root of unity. Then the natural actions of $\operatorname{Gal}(\mathbb{Q}(\zeta)/\mathbb{Q})$ on the conjugacy classes and on the irreducible characters of $G$ are permutation isomorphic.  In particular, the number of rational irreducible characters of $G$ is equal to the number of rational conjugacy classes of $G$. \end{maintheorem}

 Motivated by an element-wise criterion given in \cite{Mah18}, for a nilpotent group of class $2$ to be \cut, we give a short proof of an easy element-free criterion in Section \ref{sect:nilpotent_class_2}. This section is also supplemented with certain explicit and non-trivial examples of nilpotent \cut\ groups. 

Recently, it has been observed that an infinite simple group is always a \cut\ group \cite{BMP18}. However, this is not true for finite simple groups. In Theorem~\ref{theo:simple_cut}, we give a complete list of finite simple \cut\ groups.

\begin{maintheorem}\label{prop:simple_cut_groups} Let $G$ be a (finite) simple group. Then $G$ is \cut\ if and only if it is isomorphic to one of the following groups:
	\begin{enumerate}
		\item $C_2$,\ $C_3$,
		\item $A_7$,\ $A_8$,\  $A_9$,\ $A_{12}$,
		\item $L_2(7)$,\ $U_3(3)$,\ $U_3(5)$,\ $U_4(3)$,\ $U_5(2)$,\ $U_6(2)$,\ $S_4(3)$,\  $S_6(2)$,\ $O_8^+(2)$,
		\item $ M_{11}$,\ $M_{12}$,\ $M_{22}$,\ $M_{23}$,\  $M_{24}$,\ $Co_1$,\ $Co_2$,\ $Co_3$,\ $HS$,\ $McL$,\ $Th$,\ $M$.
	\end{enumerate} 
\end{maintheorem}

 Further, in Section~\ref{sect:local}, we explore some local properties of \cut\ groups: the impact of the \cut\ property on $p$-subgroups.
 In particular, we prove that Sylow $3$-subgroups of \cut\ groups are again \cut, for several classes of groups, see Proposition~\ref{prop:G_cut_Sylow_Cut_for_abelian_or_normal_P} and Theorem~\ref{G_cut_Sylow_Cut}:
 
 \begin{maintheorem} Let $G$ be a \cut\ group and $P \in \Syl_3(G)$. Then $P$ is also \cut, provided one of the following holds:
 	\begin{enumerate}
 		\item $P$ is abelian,
 		\item $P$ is  a normal subgroup of $G$;
 		\item $G$ is supersolvable,
 		\item $G$ is a Frobenius group,
 		\item $G$ is simple,
 		\item $G$ is of odd order and $\O_3(G)$ is abelian.
 	\end{enumerate}
 \end{maintheorem}
 
 Finally, in Section~\ref{sect:quantity}, we present surprising data on the existence of \cut\ groups. For instance, 86.62\% of all groups up to order $512$ are \cut.

\section{Notation and Preliminaries}\label{sect:preliminaries}

Throughout the article, all groups considered are finite, unless otherwise stated explicitly. Our notation is mostly standard, see \cite{Isa06, JdR15}. Let $G$ be a group and let  $x \in G$ be an element of $G$. The order of $G$ is denoted by $|G|$, the order of $x$ is denoted by $o(x)$ and $\C_{G}(x)$ denotes the centralizer of $x$ in $G$. Let $y\in G$. Then, by $x \sim y$ we mean that ``\emph{$x$ is conjugate to $y$ in $G$}'', i.e., $x^{g}:=g^{-1}x g=y$ for some $g \in G$ and $x^G$ denotes the conjugacy class of $x$ in $G$. The commutator subgroup of $G$ generated by all commutators $[x,y] = x^{-1}y^{-1}xy$, $x, y \in G$ is written as $G'$. By $H \leqslant G$ ($H \slunlhd G$), we indicate that $H$ is a subgroup (normal subgroup) of $G$.  For $H \leqslant G$,  $\N_{G}(H)$ and $[G:H]$ respectively denote the normalizer and the index of $H$ in $G$. If $p$ is a prime, then $\Syl_{p}(G)$ is the set of Sylow $p$-subgroups of $G$. By $\Irr(G)$ we denote the set of irreducible complex characters of $G$ and for $\chi\in\Irr(G)$, by $\QQ(\chi)$ we denote the field extension of the rationals generated by the values of the character $\chi$. Likewise, $\QQ(x)$ denotes the field extension of the rationals generated by the values of all the characters of $G$ at the element $x \in G$. 

An element $x \in G$ is called \emph{rational in $G$}, if $x^j \sim x$ for all $j$ coprime to $o(x)$ or, equivalently, $\QQ(x) = \QQ$, see e.g.\ \cite[Problem (2.12)]{Isa06}. Of course, a conjugacy class $x^G$ of $G$ is called \emph{rational} if it consists of elements rational in $G$. Likewise, we say a character $\chi$ of $G$ is \emph{rational} if $\QQ(\chi) = \QQ$. The group $G$ is called \emph{rational}, if every element $x \in G$ is rational in $G$, or, equivalently, every $\chi \in \Irr(G)$ is rational.

We briefly mention certain classes of groups that are closely related to \cut\ groups. A character $\chi$ of $G$ is called \emph{quadratic} if $[\QQ(\chi):\QQ] = 2$. If every irreducible character of the group $G$ is either rational or quadratic, then it is said to be \emph{quadratic rational}. Following \cite{CD10}, an element $x\in G$ is called \emph{semi-rational in $G$}, if there exists an integer $m$ such that $x^j \sim x$ or $x^j \sim x^m$ for every $j \in \ZZ$ coprime to $o(x)$ and $x$ is called \emph{inverse semi-rational in $G$}, if $m=-1$, i.e., $x^j \sim x$ or $x^j \sim x^{-1}$ for every such $j$. The group $G$ is called \emph{(inverse) semi-rational} if every element of $G$ is (inverse) semi-rational in $G$. From the definition it is clear that rational groups are inverse semi-rational.

We begin by stating some of the well-known equivalent criteria for a \cut\ group that are essential for this article.
\begin{proposition}[{see \cite[Proposition~2.2]{Bac17},~\cite[Theorem~5]{MP18}}] \label{prop:equiv_cut} For a group $G$, the following statements are equivalent:
	\begin{enumerate}
		\item\label{prop:cut} $G$ is a \cut\ group, i.e.\ $\mathcal{Z}(\mathcal{U}(\mathbb{Z} G)) = \pm \mathcal{Z}(G)$.
		\item\label{prop:conjugategenerators} $G$ is an  inverse semi-rational group, i.e.\ for every $x \in G$ and $j \in \mathbb{Z}$ with $j$ coprime to $o(x)$, $x^j \sim x$ or $x^j \sim x^{-1}$.
		\item\label{prop:wedderburndecomp} If $\QQ G \simeq \bigoplus_{k = 1}^m M_{n_k}(D_k)$ is the Wedderburn decomposition of the rational group algebra $\mathbb{Q}G$ ($m, n_k \in \mathbb{Z}_{\geqslant 1}$, $D_k$ division algebras), then for each $k$, \[ \Z(D_k) \simeq \QQ(\sqrt{-d})\] for some $d = d(k) \in \mathbb{Z}_{\geqslant 0}$.
		\item\label{prop:char} For each $\chi \in \Irr(G)$, $\QQ(\chi) = \mathbb{Q}(\sqrt{-d})$ for some $d = d(\chi) \in \mathbb{Z}_{\geqslant 0}$.
	\end{enumerate}
\end{proposition}

In particular, the notion of a \cut\ group and an inverse semi-rational group are the same. To avoid confusion, from now on, we use the term inverse semi-rational for elements of the group and the term \cut\ for a group. Note that the above proposition implies that every element of order $1$, $2$, $3$, $4$ or $6$ is inverse semi-rational in every group it is contained in; in particular all groups of exponent a divisor of $4$ or $6$ are \cut.

There are other known equivalent characterizations of \cut\ groups (for example a group $G$ is \cut\ if and only if the Whitehead group $K_1(\mathbb Z G)$ is finite), cf.\ \cite[Theorem~5]{MP18}. We will give yet another one in Proposition~\ref{prop:row_vs_column} below. The \cut\ groups are in the intersection of quadratic rational and semi-rational groups. Observe that the dihedral group of order $16$ is both, quadratic rational and semi-rational, yet it is not \cut. However, we shall see in Section~\ref{sect:quantity} that \cut\ groups constitute a big part of this intersection.

A well-known and very useful fact is that the \cut\ property is inherited by quotients. There are short proofs (e.g.\ in \cite[Lemma~4]{CD10}) of this fact using any of the equivalent statements of \cut\ groups in Proposition~\ref{prop:equiv_cut} with the exception of the very definition. We take this opportunity to provide a short $\mathbb{Z}G$-proof that avoids passing through any of the characterizations in Proposition~\ref{prop:equiv_cut}.

\begin{lemma}\label{prop:quot}
 Let $G$ be a \cut\ group and $N \slunlhd G$. Then $G/N$ is a \cut\ group.
\end{lemma}
\begin{proof}
 Let $R$ be a subring of a ring  $T$ with the same identity. 
Suppose $I\subseteq R$ and $I$ is an ideal of $T$  
Assume $\mathcal{U}(T/I)$ is a torsion group. If $\mathcal{Z}(\mathcal{U}(R))$ is torsion, then $\mathcal{Z}(\mathcal{U}(T))$ is torsion:
take $v\in \mathcal{Z}(\mathcal{U}(T))$ then $\bar{v}^n=1$ and $\bar{v}^{-n}=1$, for
some positive integer $n$,
where bars indicate reduction modulo $I$. 
Hence, $v^n-1, v^{-n}-1\in I\subseteq R$. So, $v^n, v^{-n} \in R$ and thus $v^n\in \mathcal{Z}(\mathcal{U}(R))$. Since $\mathcal{Z}(\mathcal{U}(R))$ is torsion it follows that $v\in \mathcal{Z}(\mathcal{U}(T))$ has finite order and $\mathcal{Z}(\mathcal{U}(T))$ is a torsion group.

Apply this to 
 \[ R\ =\ \mathbb{Z}G \qquad \subseteq \qquad T\ =\ \mathbb{Z}G\cdot\hat{N}\ \oplus\ \mathbb{Z}G\cdot(1 - \hat{N}) \qquad (\subseteq \mathbb{Q}G), \]
 where $\hat{N}=\frac{1}{|N|}\sum_{n\in N } n$, a central idempotent of $\mathbb{Q} G$.
Clearly $I = |N|\cdot T$ is a common ideal of $T$ and $R$, and $T/I$ is finite.
Thus $\mathcal{U}(T/I)$ is a finite group. Note that $\mathbb{Z} G \cdot \widehat{N}$ and $\mathbb{Z} [G/N]$ are isomorphic as unital rings. Hence, if $G$ is a \cut group then $\mathcal{Z}(\mathcal{U}(R))$  is finite, and, by the above, $\mathcal{Z}(\mathcal{U}(T))$ is torsion and hence all central units of \ $\mathbb{Z} G\cdot \widehat{N}\simeq \mathbb{Z} [G/N]$ are torsion. By a corollary to the Berman-Higman Theorem \cite[Corollary~7.1.4~(2)]{JdR15}, all torsion central units of an integral group ring of a finite group  are trivial. Thus $G/N$ is a \cut group.
\end{proof}

\section{Rationality}\label{sect:rationality}

It is a major problem in the representation theory of finite groups to detect dualities between conjugacy classes and irreducible characters of groups. In general the number of rational conjugacy classes and rational irreducible characters of a group $G$ do not agree\footnote{Rational conjugacy classes should not be confused with so called $\mathbb{Q}$-conjugacy classes (or $\mathbb{Q}$-classes), see \cite[p.~231]{JdR15} for a definition: the former are actual conjugacy classes of $G$ whereas the latter are in general unions of conjugacy classes of $G$, similarly for rational irreducible characters and irreducible $\mathbb{Q}$-characters. In fact, the number of $\mathbb{Q}$-classes and irreducible $\mathbb{Q}$-characters always coincide by a theorem of Artin \cite[Corollary~7.1.12]{JdR15}.}: there are groups as small as order $32$ such that the number of rational conjugacy classes exceeds the number of rational irreducible characters and vice versa; see Example~\ref{exam:classes_vs_characters} below. However, for interesting classes these numbers do coincide and it is a popular theme to detect group or representation theoretic conditions that guarantee this. For instance, G.~Navarro and P.H.~Tiep proved that this is true for groups with at most two rational irreducible characters \cite[Corollary~9.7, Theorem~A]{NT08} (these proofs require the Classification of Finite Simple Groups). Also for groups with cyclic Sylow $2$-subgroups the number of rational conjugacy classes and rational characters coincide, as shown by Navarro and J.~Tent \cite{NT10}. Note that a group with a cyclic Sylow $2$-subgroup $P$ can only be \cut\ if $|P| \leqslant 4$, by Cayley's normal $2$-complement theorem and Lemma~\ref{prop:quot}.

For a structural explanation of such phenomena one might want to consider a suitable group that acts naturally on both sets in question. Assume that $G$ is a group of exponent dividing $n$ and let $\zeta$ be a fixed primitive complex $n^{th}$ root of unity. Then the elements of the Galois group $\operatorname{Gal}(\mathbb{Q}(\zeta)/\mathbb{Q})$ are determined by $\sigma \colon \zeta \mapsto \zeta^m$ for an integer $m = m(\sigma)$ coprime to $n$, unique modulo $n$. This induces an isomorphism $\operatorname{Gal}(\mathbb{Q}(\zeta)/\mathbb{Q}) \to \U(\mathbb{Z}/n\mathbb{Z})\colon \sigma \mapsto m$. Now $\operatorname{Gal}(\mathbb{Q}(\zeta)/\mathbb{Q})$ acts naturally on the irreducible characters $\Irr(G)$ of $G$ and the conjugacy classes $\ccl(G)$ of $G$ by \[ \chi^\sigma = \sigma^{-1} \circ \chi, \qquad (x^G)^\sigma = (x^{m})^G, \qquad \chi \in \Irr(G),\ x \in G,\ \sigma \in \operatorname{Gal}(\mathbb{Q}(\zeta)/\mathbb{Q}).\] Note that the fixed points of these actions are precisely the rational irreducible characters and the rational conjugacy classes, respectively. One might wonder under which conditions these actions are essentially the same, i.e., in which situations they are permutation isomorphic: recall that two actions of a group $\Gamma$ on sets $X$ and $Y$ are called \emph{permutation isomorphic}, if there is a bijection $\alpha\colon X \to Y$ such that $\alpha(x^g) = \alpha(x)^g$ for all $x \in X$ and all $g \in \Gamma$. The above actions are permutation isomorphic for $p$-groups of odd order (follows from Brauer's permutation lemma \cite[Theorem (6.32)]{Isa06}) and also for groups all whose Sylow subgroups are abelian \cite{Bro71}. Here we show that this also happens for \cut\ groups. In view of the group theoretic characterization of \cut\ groups in Proposition~\ref{prop:equiv_cut}\eqref{prop:conjugategenerators}, this can be seen as a contribution to \cite[(14.4)~Problem]{Nav10}. 

\begin{theorem}\label{theo:perm_iso_Galoi_actions} Let $G$ be a \cut\ group of exponent dividing $n$ and let $\zeta$ be a primitive $n^{th}$ root of unity. Then the natural actions of $\operatorname{Gal}(\mathbb{Q}(\zeta)/\mathbb{Q})$ on the conjugacy classes and on the irreducible characters of $G$ are permutation isomorphic.  In particular, the number of rational irreducible characters of $G$ is equal to the number of rational conjugacy classes of $G$.
\end{theorem}

\begin{proof} We start with the following group theoretic fact. \\
\emph{Claim. }Assume a (finite) group $\Gamma$ acts on two finite sets $X$ and $Y$ of the same size and every orbit of $\Gamma$ on both sets has length at most $2$. If every element of $\Gamma$ has the same number of fixed points on $X$ as it has on $Y$, then the actions are permutation isomorphic.

Denote by $f(\sigma)$ the number of fixed points of $\sigma \in \Gamma$ on $X$ or, equivalently, on $Y$. By a classical result on permutation actions, see \cite[Lemma (13.23)]{Isa06}, it suffices to show that each $\Delta \leqslant \Gamma$ has the same number of fixed points on $X$ as it has on $Y$ to conclude that the actions are permutation isomorphic.  Arguing by induction on $|\Gamma|$, it is enough to show that $\Gamma$ has the same number of fixed points on $X$ as it has on $Y$. 
By Burnside's lemma \cite[Theorem~3.22]{Rot95},
the number $k$ of orbits of $\Gamma$ on $X$ or on $Y$ is
  \[k = \frac{1}{|\Gamma|}\sum_{\sigma \in \Gamma} f(\sigma). \] Thus the number of fixed points on both sets is just $2k - |X| = 2k - |Y|$. This proves the claim.
  
Set $\Gamma = \operatorname{Gal}(\mathbb{Q}(\zeta)/\mathbb{Q})$. By Proposition~\ref{prop:equiv_cut}\eqref{prop:char}, we have for each $\chi \in \Irr(G)$ that $\QQ(\chi) = \QQ(\sqrt{-d})$ for some $d \in \ZZ_{\geqslant 0}$. So the length of each orbit of the natural action of $\Gamma$ on $X =\Irr(G)$ is at most $2$. 
By Proposition~\ref{prop:equiv_cut}\eqref{prop:conjugategenerators}, the length of each orbit of the natural action of $\Gamma$ on $Y = \ccl(G)$ is at most $2$. We want to employ Brauer's permutation lemma \cite[Theorem (6.32)]{Isa06} to obtain that each $\sigma \in \Gamma$ has the same number of fixed points on $X$ as it has on $Y$. To do this we need to verify that $\chi^\sigma(x^\sigma) = \chi(x)$ for all $\chi \in \Irr(G)$ and all $x \in G$. This can be checked by the following  calculation. Let $D$ be a representation with character $\chi$. Then $D(x)$ is conjugate to $\operatorname{diag}(\rho_1, ..., \rho_d)$, a diagonal matrix with $n^{th}$ roots of unity $\rho_1, ..., \rho_d$ on its diagonal. Let $\sigma \in \Gamma$ and $m = m(\sigma)$ be as defined above. Then \begin{align*} \chi^\sigma(x^\sigma) & = \sigma^{-1} \circ \operatorname{Tr} D(x^m) = \sigma^{-1} \circ \operatorname{Tr} D(x)^m = \sigma^{-1} \circ \operatorname{Tr} \operatorname{diag}(\rho_1^m, ..., \rho_d^m) \\ & = \sigma^{-1}(\rho_1^m + ... + \rho_d^m) = \rho_1 + ... + \rho_d = \operatorname{Tr} \operatorname{diag}(\rho_1, ..., \rho_d) = \operatorname{Tr} D(x) \\ & = \chi(x). \end{align*}
Consequently, the natural actions of $\Gamma$ on $X$ and on $Y$ are permutation isomorphic by the above claim. 

Since the actions of $\Gamma$ on the irreducible characters and on the conjugacy classes are permutation isomorphic, they have the same number of fixed points, so the claim about the rational valued characters and conjugacy classes follows.
\end{proof}

Let $G$ be a group of exponent $n$ and $\Gamma = \operatorname{Gal}(\mathbb{Q}(\zeta)/\mathbb{Q})$ for a primitive complex $n^{th}$ root of unity $\zeta$. Denote by $N \slunlhd \Gamma$ the kernel of the natural action of $\Gamma$ on $\Irr(G)$ (or on $\ccl(G)$). If $G$ is semi-rational or quadratic rational, then $\Gamma/N$ is an elementary abelian $2$-group (see e.g.\ \cite[Lemma~7]{CD10}).  If $G$ is a solvable semi-rational group (solvable quadratic rational group, respectively) then $|\Gamma/N|$ is bounded by $2^6$ (by $2^7$, respectively) by \cite[Theorem~B \& Corollary~8]{Ten12}. In particular, for a solvable \cut\ group $G$ one obtains that $\Gamma/N$ has order dividing $2^5$ by adapting the proof of \cite[Theorem~B]{Ten12} and using that $|G|$ is divisible by at most $4$ different primes by \cite[Theorem~1.2]{Bac17}. Many solvable \cut\ groups such that $|\Gamma/N| = 8$ can be found in the library of small groups in \textsf{GAP} \cite{GAP4}. For instance, the groups with SmallGroupIDs \texttt{[144, 58]}, \texttt{[192, 718]}, and \texttt{[960, 11363]}; the sets of fields of character values of their irreducible characters are, respectively, $\{\QQ, \QQ(\sqrt{-1}), \QQ(\sqrt{-2}), \QQ(\sqrt{-3}) \}$, $\{\QQ, \QQ(\sqrt{-1}), \QQ(\sqrt{-3}), \QQ(\sqrt{-6}) \}$, and $\{\QQ, \QQ(\sqrt{-1}), \QQ(\sqrt{-3}), \QQ(\sqrt{-15}) \}.$ We do not know of an example of a solvable \cut\ group such that $|\Gamma/N|$ exceeds $2^3$. Note that alternating groups are both, semi-rational and quadratic rational. So if we drop the solvability assumption, then \cite{RT95} shows that $\Gamma/N$ can get arbitrarily large for semi-rational and quadratic rational groups. However, we do not know whether or not the order $|\Gamma/N|$ is bounded in the case of arbitrary \cut\ groups (cf.\ \cite[Question~1]{Bac19}). 

We continue with yet another characterization of \cut\ groups based on fields of character values that is dual to the one in Proposition~\ref{prop:equiv_cut}\eqref{prop:char}: a group $G$ is \cut\ if, and only if, for all irreducible characters $\chi$ of $G$, $\QQ(\chi)$ is $\QQ$ or imaginary quadratic.
\begin{proposition}\label{prop:row_vs_column} Let $G$ be a group. The following statements are equivalent:
	\begin{enumerate}
		\item $G$ is a \cut\ group.
		\item\label{it:prop_columns} For each $x \in G$, $\QQ(x) = \mathbb{Q}(\sqrt{-d})$ for some $d = d(x) \in \mathbb{Z}_{\geqslant 0}$.
	\end{enumerate} 
\end{proposition}
\begin{proof} Assume first that $G$ is a \cut\ group. Let $x \in G$, so that $x$ is inverse semi-rational in $G$ and in particular semi-rational in $G$. From \cite[Lemma~1]{Ten12}, it follows that $[\QQ(x) : \QQ] \leqslant 2$. Now, if $x$ is not rational, then for some $\chi \in \Irr(G)$, $\QQ\neq\QQ(\chi(x))\subseteq \QQ(x)$ and degree considerations yield $\QQ(\chi(x))=\QQ(x)$. Furthermore,   $\QQ\neq\QQ(\chi(x))\subseteq \QQ(\chi)$ and by Proposition \ref{prop:equiv_cut}, $\QQ(\chi)=\QQ(\sqrt{-d})$, for some $d > 0$. Consequently, \eqref{it:prop_columns} follows. 

Conversely, assume that for each $x \in G$ we have $\QQ(x) = \QQ(\sqrt{-d})$, $d \geqslant 0$. Fix $x \in G$. Again by \cite[Lemma~1]{Ten12}, $x$ is semi-rational in $G$, so that there is $m \in \mathbb{Z}$ such that
\begin{equation}\label{conjugacy}\tag{C} \text{for every integer } j \text{ coprime to } o(x): \qquad
 x^{j} \sim x ~\text{ or }~ x^{j} \sim x ^{m}.
\end{equation}
Assume first $x \sim x^{-1}$. Then for every $\chi \in \Irr(G)$ we have $\chi(x) = \chi(x^{-1}) = \overline{\chi(x)}$, where bar indicates complex conjugation. It follows that $\chi(x) \in \mathbb{Q}(\sqrt{-d}) \cap \mathbb{R} = \mathbb{Q}$ for all $\chi \in \Irr(G)$. Thus $\QQ(x)=\QQ$, i.e., $x$ is rational in $G$. Otherwise $x$ is not conjugate to $x^{-1}$ in $G$, and in view of \eqref{conjugacy}, $x^{-1}$ is conjugate to $x^m$ and we actually may chose $m = -1$ in \eqref{conjugacy}. Hence in any case $x$ is inverse semi-rational in $G$ and $G$ is a \cut\ group (Proposition \ref{prop:equiv_cut}).
\end{proof}

Note that Theorem~\ref{prop:rational_classes_vs_characters} cannot be extended to the slightly larger class of semi-rational groups nor to the class of quadratic rational groups. A reason for this is that the nice symmetry  between the rows and the columns of the character table in the characterization of \cut\ groups (Propositions~\ref{prop:equiv_cut}\eqref{prop:char} and \ref{prop:row_vs_column}) does not hold true for these classes of groups. Indeed, by \cite[Lemma~1]{Ten12} a group is semi-rational if and only if $[\QQ(x):\QQ] \leqslant 2$ for all $x \in G$. However the degree of the dual field extensions $\QQ(\chi)$ over $\QQ$ can get large, e.g.\ for the semi-rational group $G$ with SmallGroupID \texttt{[384, 3283]} in \textsf{GAP} \cite{GAP4} there is $\chi \in \Irr(G)$ such that $\QQ(\chi) = \QQ(\zeta)$, for a primitive $24^{th}$ root of unity $\zeta$. Similarly, one can find examples of a quadratic rational groups $G$ such that the fields $\mathbb{Q}(x)$, $x\in G$ get large. The following example shows that already for groups of order $32$ the actions of the Galois group on the rows and columns of the character table need not be permutation isomorphic and even that the numbers of their fixed points might differ.

\begin{example}\label{exam:classes_vs_characters} 
In \cite[Section~6]{Ten12}, Tent  defined a semi-rational group $G$ and a quadratic rational group $H$ (the latter originally constructed by E.~Dade, see \cite[Example~3.6]{Bro71}) as follows:
\begin{align*} G = & \langle\ a, b, c \ | \ a^2 = b^2 = c^8 = 1,\ b^c = b,\ b^a = bc^4,\ c^a = c^3\  \rangle,  \\
H = & \langle\ a, b, c \ | \ a^2 = b^2 = c^8 = 1,\ b^c = b,\ b^a = b,\ c^a = bc^3\  \rangle  \end{align*}
(having \textsf{GAP}  SmallGroupIDs \texttt{[32, 42]} and \texttt{[32, 9]}, respectively). If we denote by $\Irr_\QQ(G)$ and $\ccl_\QQ(G)$, the set of rational irreducible characters and rational conjugacy classes of $G$, respectively, then we have $|\Irr_\QQ(G)| = 10 > |\ccl_\QQ(G)| = 8$ and $|\Irr_\QQ(H)| = 6 < |\ccl_\QQ(H)| = 8$. These are examples of smallest possible order for which these quantities differ. 
\end{example}

\section{Nilpotent \cut\ groups}\label{sect:nilpotent_class_2}

As mentioned in the introduction, it is well known that an abelian group is \cut, if and only if its exponent divides 4 or 6. As the \cut\ property is quotient closed (Lemma~\ref{prop:quot}) and the center of a \cut\ group is again a \cut\ group, we necessarily have for each \cut\ group $G$:
\begin{equation}  \text{for all } N \slunlhd G:\qquad \exp (\Z(G/N)) \mid 4\quad \text{ or }\quad \exp(\Z(G/N)) \mid 6 \tag{N}\label{N},\end{equation}
where $\exp(H) $ denotes the exponent of a group $H$.
However, rarely condition \eqref{N} is sufficient for $G$ to be a \cut\ group. For example, \eqref{N} holds for all non-abelian simple groups, but not all simple groups are \cut\ (see Section \ref{sect:simple_cut_groups}). Also, dihedral groups of order $2p$, $p$ a prime, satisfy \eqref{N}, but these groups are not \cut\ for $p \geqslant 5$ (and are of derived length $2$). 
In \cite{BMP17} it was noted that a nilpotent \cut\ group is always a $\{2,3\}$-group and an element wise criterion for a nilpotent group to be a \cut\ group was provided. In the case of nilpotency class $2$, the criterion was refined in \cite{Mah18}. For a nilpotent group $G$ of class $2$ to be a \cut\ group, we provide a characterization on the quotients of $G$. We prove that, for such groups, condition \eqref{N} is sufficient for $G$ to be a \cut\ group. 

After we had finished writing the first version of this article we were informed by V.~Bovdi that such a characterization was already obtained by Z.~Pata\u\i\ and A.~Bovdi \cite{Pat78,Bov87}.
 The formulation of \cite[Theorem~8.2]{Bov87} gave the inspiration for the new group theoretic proof below that is notably shorter than the ring theoretic proof presented in \cite{Bov87}. As it seems to be hard to get hold of the above sources, being available in Ukrainian and Russian only and since the result seems to have stayed unnoticed by many experts, we want to take the opportunity to popularize this theorem together with our concise proof below.

\begin{theorem}[Pata\u\i]\label{nilpotentClass2}
	Let $G$ be a nilpotent group of class at most $2$. Then $G$ is \cut\ if, and only if, the exponent of $\Z(G/N)$ divides $4$ or $6$, for all $N \slunlhd G$. Actually it is sufficient to deal with normal subgroups $N$ of the form $[g, G]$.
\end{theorem}

\begin{proof} Clearly the condition is necessary. For sufficiency,  assume that $G$ is a finite nilpotent group of class at most $2$ and the exponent of $\Z(G/N)$ divides 4 or 6, for all $N \slunlhd G$. The nilpotency of $G$ yields that $G$ is a $\{2, 3\}$-group and $G = P_2 \times P_3$, where $P_2$ and $P_3$ denote its Sylow $2$- and $3$-subgroups, respectively. In view of \cite[Theorem~3]{Mah18}, $G$ is a \cut\ group if, and only if both $P_{2}$ and $P_{3}$ are \cut\ groups; and, moreover, $P_{2} $ is rational, if $P_{3}$ is non-trivial. 
Since $G$ is of class at most $2$, $[h, G] = \{[h, g] \colon g \in G\}$ is a normal subgroup of $G$, for any $h\in G$. 
We first check that $P_{2}$ is a \cut\ group. By \cite[Corollary~3]{Mah18} this is equivalent with $x^4 \in [x,P_2]=[x, G]$ for all $x \in P_2$. But this follows from the assumptions, since the image of $x$ is contained in $\Z(G/[x, G])$. Similarly, $P_{3}$ is also a \cut\ group. Hence, it only remains to show that if $P_{3}$ is non-trivial, then $P_{2}$ is rational. Assume that $P_3 \not= \{1\}$ and take $x \in P_2$. 
Clearly $[x, G] =[x,P_2]\leqslant P_2$ and $3$ divides the order of $\Z(G/[x, G])$. Then, by the assumption, $x^6 \in [x, G]$ and thus $x^2 \in [x, G]$. 
Hence, for any odd integer
$j=2v+1$ we have  $x^{j}=x(x^{2})^{v}\in x [x,G]=x[x,P_2]$ and thus $x^j\in x^{P_2}$.
 So, $x$ is rational in $P_2$, as desired.
\end{proof}

\begin{remark}
The following two groups of nilpotency class $3$ illustrate that condition \eqref{N} is no longer sufficient to conclude that $G$ is \cut\ for nilpotent groups of nilpotency class exceeding $2$. First, let $G = D_{16}$, the dihedral group of order $16$. Then we have $\exp (\Z(G/N) )\mid 2$ for every normal subgroup, yet the group is not \cut. Similarly, if \[G = \langle\ a, b, c\ \mid\ a^9 = b^3 = c^3 = 1,\ b^a = b,\ a^c = ab,\ b^c = a^3b\ \rangle \simeq (C_9 \times C_3) \rtimes C_3, \] (SmallGroupID \texttt{[81, 8]}) all centers of quotients of $G$ are of exponent $3$, but this group is not \cut.
\end{remark}

There is a plethora of nilpotent \cut\ groups of small exponent, see Proposition~\ref{rem:negligible} and its proof. Here we provide further examples of nilpotent \cut\ groups, in particular some of arbitrary large exponent and derived length (and hence nilpotency class).

\begin{example}\label{ex:nilpotent}
		\begin{enumerate} 
 \item\label{ex:Sylow_Sym}\textbf{Sylow $p$-subgroups of $S_n$, for all $n$, $p \in \{2,3\}$:}
  The Sylow $2$- and $3$-subgroups of symmetric groups $S_n$ are rational and \cut, respectively, as follows from \cite[Remark~14]{CD10}. We provide some explanation here. The structure of the Sylow $p$-subgroups of the symmetric group $S_n$ is well known, see \cite[Theorem~7.27]{Rot95} and the discussion following it.
 Denote by $Q_n$ a Sylow $2$-subgroup of $S_n$. First assume that $n$ is a power of $2$. Then $Q_n = ((C_2 \wr C_2) ... \wr C_2) \wr C_2$ is an iterated wreath product of cyclic groups of order $2$. Clearly, $C_2$ is rational and so is $Q_n$ by \cite[Proposition~3.5]{Heg05}.  Now if $n$ is arbitrary, consider its base $2$ expansion $n = a_k\cdot 2^k + a_{k-1}\cdot 2^{k-1} + ... + a_1\cdot 2 + a_0$ with each $a_\ell \in \{0, 1\}$. Then the Sylow $2$-subgroup of $S_n$ is \begin{equation}\label{decomposition}\tag{D} Q_n = Q_{2^k}^{a_k} \times Q_{2^{k-1}}^{a_{k-1}} \times ... \times Q_2^{a_1}. \end{equation}  Since the direct product of rational groups is rational again, $Q_n$ is a rational group for all $n$.\\
 Similarly, the Sylow $3$-subgroups $R_n$ of $S_n$ are iterated wreath products of $C_3$ if $n$ is a power of $3$. Since $\QQ(\chi) \subseteq \QQ(\sqrt{-3})$ for all $\chi \in \Irr(C_3)$ we also have that for these iterated wreath products $\QQ(\psi) \subseteq \QQ(\sqrt{-3})$ for all $\psi \in \Irr(R_n)$ by repeated application of \cite[Proposition~3.5]{Heg05} (see also the remark before the definition preceding that proposition), so they are \cut\ groups. In the general setting, the Sylow $3$-subgroups $R_n$ of $S_n$ are also \cut\ as direct products of these iterated wreath products by \cite[Corollary~1]{BMP17}. Note that the exponent and derived length of $Q_n$ and $R_n$ grow beyond all limits.
 \item\textbf{Sylow $p$-subgroups of $A_n$, for all $n$, $p \in \{2,3\}$:} Also the Sylow $2$- and $3$-subgroups of alternating groups $A_n$ are rational and \cut, respectively. Since $A_n$ has index $2$ in $S_n$, the Sylow $3$-subgroups of $A_n$ are \cut\ by {\eqref{ex:Sylow_Sym}} above. We now show that the Sylow $2$-subgroups of alternating groups are rational. For this, let $Q_n$ be again a Sylow $2$-subgroup of $S_n$. Then $P_n = Q_n \cap A_n$ is a Sylow $2$-subgroup of $A_n$. Let $x \in P_n$ and let $j$ be an odd integer. We need to show that there is $y \in P_n$ such that $x^y = x^j$. Write $x = x_k\cdot...\cdot x_1$ with $x_\ell \in Q_{2^\ell}$ according to the decomposition in \eqref{decomposition}. Now $x_\ell \in Q_{2^\ell}$ is conjugate by a suitable $y_\ell \in Q_{2^\ell}$ to $x_\ell^j$, since $Q_{2^\ell}$ is rational. Note that we may choose $y_\ell$ as an odd permutation (in the case $y_\ell$ is even and $x_\ell \not= 1$, decompose $x_\ell = t_1 \cdot ... \cdot t_s$ in non-identity disjoint cycles of $Q_{2^\ell}$ of $2$-power length, then $y_\ell'= t_1 y_\ell \in Q_{2^\ell}$ is an odd permutation such that $x_\ell^{y_\ell'} = x_\ell^j$). Then for $y \in P_n$, the product of all these $y_\ell$, we have $x^y = x^j$.\\
 If $n = 4m+2$ or $n = 4m+3$ for some $m \in \mathbb{Z}_{\geqslant 1}$, then $S_{4m}$ embeds into $A_{4m+2}$ and $A_{4m+3}$ and the image has odd index $(2m+1)(4m+1)$ and $(2m+1)(4m+1)(4m+3)$, respectively, so their Sylow $2$-subgroups are isomorphic. 
 However, in all the other cases order considerations show that the Sylow $2$-subgroups of $A_n$ can never be isomorphic to a Sylow subgroup of a symmetric group, so they are genuinely different examples of rational nilpotent groups of large derived length and exponent.
\item\textbf{Some Sylow $p$-subgroups of linear groups:} \label{ex:Sylow_PSL} To indicate that it might be difficult to decide whether a group is \cut\ or not, even for intensively studied groups, we consider for instance the group of unipotent upper triangular matrices over finite fields.  Let $p$ be a prime and let $n$ and $f$ be positive integers. Denote by $P$ the Sylow $p$-subgroup of $\GL(n, p^f)$ consisting of unipotent upper triangular matrices. Note that $P$ is also isomorphic to a Sylow $p$-subgroup of $\PSL(n,p^f)$. It was believed for a long time and supported by experimental calculations that for $p = 2$ the group $P$ is a rational group for all $n$. Only in 1998, I.M.~Isaacs and D.~Karagueuzian, in \cite{IK98,IK98_2}, gave an example of a matrix $A \in P$ that is not conjugate to its inverse for $n = 13$ and $p^f = 2$. Actually, the Sylow $p$-subgroups of $\GL(n, p^f)$ are rational if, and only if, $p = 2$ and $n \leqslant 12$, see \cite[Introduction]{Mar12}. Also, the Sylow $3$-subgroups of $\GL(n, 3^f)$ are \cut\ for all $n \leqslant 12$, see \cite[Introduction]{Mar12}. However, it seems to be still unknown when exactly the Sylow $p$-subgroups of $\GL(n, p^f)$ are \cut. If they are, then ($p = 2$ and $n \leqslant 48$) or ($p = 3$ and $n \leqslant 18$) by \cite[Theorem]{Mar12}.
 \end{enumerate}
\end{example}

For some other simple groups it is verified in the proof of Theorem~\ref{G_cut_Sylow_Cut}\eqref{item:simple_3} that their Sylow $3$-subgroups are \cut.

\begin{question}\label{que:Sylow_classical}
 Which Sylow $2$- and $3$-subgroups of classical groups or groups of Lie type are \cut? 
\end{question}

\begin{remark}\label{re:emb_nilp} It is well known that every $\{2,3\}$-group can be embedded into a rational $\{2,3\}$-group, see e.g.~\cite[Proposition 1]{Gow76}. Not every nilpotent $\{2,3\}$-group can be embedded into a nilpotent rational group though, for trivial reasons (every non-trivial rational group has even order). However, Example~\ref{ex:nilpotent}\eqref{ex:Sylow_Sym} shows that every nilpotent $\{2,3\}$-group can be embedded into a nilpotent \cut\ group.

The set $\pi = \{2,3\}$ above is optimal for the embedding property for solvable $\pi$-groups into rational solvable groups: every rational solvable group is a $\{2,3,5\}$-group, by a classical result of R.~Gow \cite{Gow76}, and has an elementary abelian Sylow $5$-subgroup by work of P.~Heged\H{u}s \cite{Heg05}, so the structure of the Sylow $5$-subgroup is a proper obstruction. 
\end{remark}

However, for solvable \cut\ groups, we show next that the obstruction in Remark~\ref{re:emb_nilp} is no longer an obstruction for the corresponding embedding question. Recall \cite[Theorem~1.2]{Bac17} that the order of a solvable \cut\ group can only be divisible by the primes $2$, $3$, $5$ and $7$.

\begin{proposition} Every $5$- and every $7$-group can be embedded into a solvable \cut\ group. \end{proposition}
\begin{proof} We work the proof out for $p = 5$. Consider $P = \langle (1,2,3,4,5)\rangle \leqslant S_5$, a Sylow $5$-subgroup of $S_5$ and set \[W_5 = \N_{S_5}(P) = \langle\ (1,2,3,4,5), (2,3,5,4)\ \rangle \leqslant S_5, \] its normalizer in $S_5$.  Then $W_5 \simeq C_5 \rtimes C_4$, the Frobenius group of order $20$. Note that $\mathbb{Q}(\chi) \subseteq \mathbb{Q}(i)$ for all $\chi \in \Irr(W_5)$ and hence $W_5$ is a solvable \cut\ group.  Define $W_{5^{n+1}} = W_{5^n} \wr W_5$, $n\in \mathbb{Z}_{\geqslant1}$. Induction and \cite[Proposition~3.5]{Heg05} show that $W_{5^n}$ is a \cut\ group for all $n\in \mathbb{Z}_{\geqslant1}$ which is clearly solvable as wreath product of solvable groups. Note that the Sylow $5$-subgroup of $W_{5^n}$ is isomorphic to a Sylow $5$-subgroup of $S_{5^n}$. Since every $5$-group can be embedded into a symmetric group of sufficiently large degree, by Cayley's theorem, it can also be embedded into some $S_{5^n}$ and hence into the solvable \cut\ group $W_{5^n}$. 

For $p = 7$ one can argue similarly using $W_7 = \N_{S_7}(\ \langle (1,2,3,4,5,6,7)\rangle\ ) \leqslant S_7$, which is isomorphic to the Frobenius group of order $42$.
\end{proof}

\begin{question}\label{que:solvable_emb} Can every solvable $\{2,3,5,7\}$-group be embedded into a solvable \cut\ group?
\end{question}

\section{Simple \cut\ groups} \label{sect:simple_cut_groups}
 
So far, the properties of solvable \cut\ groups have been explored. A complete classification of finite metacyclic \cut\ groups is given in \cite[Theorem~5]{BMP17}. A description of Frobenius \cut\ groups can be found in \cite[Theorem~1.3]{Bac17} (it turns out that those groups are always solvable). In this section, we give a complete classification of finite simple \cut\ groups based on work of S.H.~Alavi and A.~Daneshkhah.

\begin{theorem}\label{theo:simple_cut} Let $G$ be a simple group. Then $G$ is \cut\ if and only if it is isomorphic to one of the following groups:
\begin{enumerate}
 \item $C_2$,\ $C_3$,
 \item $A_7$,\ $A_8$,\  $A_9$,\ $A_{12}$,
 \item $L_2(7)$,\ $U_3(3)$,\ $U_3(5)$,\ $U_4(3)$,\ $U_5(2)$,\ $U_6(2)$,\ $S_4(3)$,\  $S_6(2)$,\ $O_8^+(2)$,
 \item $ M_{11}$,\ $M_{12}$,\ $M_{22}$,\ $M_{23}$,\  $M_{24}$,\ $Co_1$,\ $Co_2$,\ $Co_3$,\ $HS$,\ $McL$,\ $Th$,\ $M$.
 \end{enumerate} \end{theorem}

\begin{proof} The simple semi-rational groups are classified in \cite[Theorem~1.1]{AD17} using the Classification of Finite Simple Groups. (Note that the groups in the statement of \cite[Theorem~1.1]{AD17} are not entirely correctly listed: the groups ${}^3D_4(2)$, ${}^3D_4(3)$, ${}^2B_2(8)$, ${}^2B_2(32)$, ${}^2G_2(27)$ and the Tits group ${}^2F_4(2)'$ are there by accident and the group $G_2(4)$ is missing; yet the proofs in the article \cite{AD17} correctly identify the former groups as not being semi-rational and the latter group as semi-rational). Disregarding the alternating groups this is a finite list. The alternating groups that are \cut\ are described in \cite{Fer04,AKS08}. An inspection of the character tables of the remaining groups (for example in ATLAS \cite{ATLAS} or \textsf{GAP}) using Proposition~\ref{prop:equiv_cut} reveals that of those exactly the groups listed above are \cut\ groups. \end{proof}

Note that in contrast to the case of finite simple groups, every infinite simple group $I$ is a \cut\ group i.e., has the property that $\Z(\U(\ZZ I)) = \pm \Z(I)$, see \cite[Examples following Theorem~5]{BMP18}.

\section{Local properties of \cut\ groups}\label{sect:local}

It was conjectured for a long time that being rational for $2$-elements is governed by the Sylow $2$-subgroup of a group. More precisely, already in Kletzing's book from 1984 \cite[p.~13]{Kle84}  it is referred to as a ``long standing conjecture''  that the Sylow $2$-subgroup of a rational group is again rational (recall that every non-trivial rational group is of even order and hence has a non-trivial Sylow $2$-subgroup). Eventually in 2012, Isaacs and Navarro provided rational groups of order $2^9\cdot 3$ with Sylow $2$-subgroups that are not rational \cite{IN12}. However, they also proved the conjecture for solvable groups with Sylow $2$-subgroups of nilpotency class at most $2$ \cite[Theorem~A]{IN12}.

Recall that every non-trivial \cut\ group has order divisible by $2$ or $3$ \cite[Theorem~1]{BMP17} and one might wonder which properties are determined locally (i.e., in $p$-subgroups or their normalizers). It is not hard to find \cut\ groups that have Sylow $2$-subgroups that are not \cut, e.g.\ the groups with SmallGroupIDs \texttt{[384, 18033]} and  \texttt{[384, 18040]} of order $2^7\cdot 3$. However, for the prime $3$ things seem to behave differently. The following lemma shows that being inverse semi-rational for $3$-elements is indeed a somehow local property.
\begin{lemma}\label{lem:cut-3-local} Let $G$ be a group and let $x \in G$ be a $3$-element. Then $x$ is inverse semi-rational in $G$ if and only if $x$ is inverse semi-rational in $P$ for some $P \in \Syl_3(G)$. \end{lemma}
\begin{proof} Suppose $y$ is an element of a group $Y$. Then $B_Y(y): = \N_Y(\langle y \rangle)/\C_Y(y)$ can naturally be identified with a subgroup of $\Aut(\langle y \rangle)$, the automorphism group of $\langle y\rangle$. By \cite[Lemma 5]{CD10}, the element $y$ is inverse semi-rational in $Y$ if and only if $B_Y(y) \langle \tau \rangle = \Aut(\langle y \rangle)$, where $\tau \colon \langle y \rangle \to \langle y \rangle ~(w \mapsto w^{-1})$ denotes the inversion automorphism of $\langle y \rangle$. Recall that the automorphism group of the cyclic group $C_{3^f}$ of order $3^f$ is cyclic $(\Aut(C_{3^f}) \simeq C_2 \times C_{3^{f-1}})$, and $2$ is a primitive root modulo $3^f$, i.e., $2$ generates the unit group of the ring of integers modulo $3^f$. Hence, the $3$-element $x$ is inverse semi-rational in an ambient group if and only if $x$ is conjugate to $x^4$ in that group.  Assume that the $3$-element $x \in G$ of order $3^f$ is inverse semi-rational in $G$. Then there exists $g \in \N_G(\langle x \rangle)$ such that $x^g = x^4$. Since the automorphism $w \mapsto w^4$ of $\langle x \rangle$ is of $3$-power order, we may replace $g$ by its $3$-part, if necessary, and can assume that $g$ is also a $3$-element. Now, since $g$ normalizes $\langle x \rangle$, the subgroup $S = \langle x, g \rangle$ is a $3$-subgroup of $G$. By Sylow's theorem, $S$ is contained in a Sylow $3$-subgroup $P$ of $G$ and $x$ is inverse semi-rational in $P$. The other implication is clear. \end{proof}

The above lemma is clearly false for $p$-elements, $p \geqslant 5$.\\

It may be observed that if $G$ is a \cut\ group and  $P \in \Syl_3(G)$ then $P$ is \cut\ if and only if for all $x \in P$ and for all $S\in  \Syl_3(G)$ containing $x$, $x$ is inverse semi-rational in $S$.
Lemma~\ref{lem:cut-3-local} asserts that each $3$-element of a \cut\ group is inverse semi-rational in at least one Sylow $3$-subgroup. From the next lemma one can deduce that it is inverse semi-rational in all Sylow $3$-subgroups, in certain cases.

\begin{lemma}\label{prop:centers_of_sylow_subgroups} Let $G$ be a \cut\ group and $P \in \Syl_p(G)$ for some prime $p$. Then $\exp \Z(P) \mid p$, if $p$ is odd and $\exp \Z(P) \mid 4$, if $p = 2$.  \end{lemma}

\begin{proof} Let $x \in \Z(P)$ be an element of order $p^f$. Then $P \leqslant \C_G(x)$ and hence $B_G(x) = \N_G(\langle x \rangle)/\C_G(x)$ has an order not divisible by $p$. Note that, in case $p$ is odd, $\Aut(\langle x \rangle) \simeq C_{p^{f-1}} \times C_{p-1}$ and $\Aut(\langle x \rangle) \simeq C_{2^{f-2}} \times C_{2}$, if $p = 2$. Further, by \cite[Lemma 5]{CD10}, we have that $[\Aut(\langle x \rangle) : B_G(x)] \leqslant 2$ and thus the result follows.\end{proof}

In particular, abelian Sylow $p$-subgroups of \cut\ groups are elementary abelian for odd primes $p$ and abelian Sylow $2$-subgroups of \cut\ groups are of exponent dividing $4$.\\

Let $p$ be a prime. For a group $X$ denote by $\O_{p}(X)$ and $\O_{p'}(X)$ the maximal normal $p$-subgroup of $X$ and maximal normal $p'$-subgroup of $X$, respectively. For a group $G$ we can define the upper $p$-series by \[1\ \leqslant\ \O_{p'}(G)\ \leqslant\ \O_{p',p}(G)\ \leqslant\  \O_{p',p, p'}(G)\ \leqslant\ ...\ , \] where $\O_{p',p}(G)/\O_{p'}(G) = \O_{p}(G/\O_{p'}(G))$, $\O_{p',p, p'}(G)/\O_{p',p}(G) =  \O_{p'}(G/\O_{p',p}(G))$ and so on (alternating between $p$ and $p'$). If this series terminates in $G$, then $G$ is called \emph{$p$-solvable}. In this case, the \emph{$p$-length} of $G$ is defined to be the number of occurrences of the symbol $p$ in the subscript of the first group in the upper $p$-series of $G$ that is equal to $G$. We prove the following:

\begin{proposition}\label{p-length_1} Let $G$ be a $3$-solvable \cut\ group of $3$-length at most $1$, and let $P \in \Syl_3(G)$. Then $P$ is \cut. \end{proposition}

\begin{proof} Since $G$ has $3$-length at most $1$, $G$ has normal subgroups $1 \leqslant M \leqslant N \leqslant G$ such that $M$ and $G/N$ are $3'$-groups and $N/M$ is a $3$-group. Let $x \in P$. Then $x$ is inverse semi-rational in $G$ and hence also in some Sylow $3$-subgroup $Q$ of $G$ by Lemma~\ref{lem:cut-3-local}. Since $N$ is a normal subgroup of $3'$-index, $Q$ is contained in $N$ and hence $x$ is inverse semi-rational in $N$. Note that $M$ is a normal complement for $P$ in $N$, that is, $N = MP$ is the semi-direct product of the normal subgroup $M$ with $P$ and the restriction to $P$ of the projection map $\pi \colon N = MP \twoheadrightarrow P$ is the identity map on $P$. As $x$ is inverse semi-rational in $N$, there is $n \in N$ such that $x^n = x^4$. Write $n = my$, $m \in M$, $y \in P$. Then $x^y = x^4$, since $\pi(n) = y$. Hence, $x$ is inverse semi-rational in $P$, implying that $P$ is \cut. \end{proof}

The following proposition follows from Lemma~\ref{prop:centers_of_sylow_subgroups} and Proposition~\ref{p-length_1}:

\begin{proposition}\label{prop:G_cut_Sylow_Cut_for_abelian_or_normal_P}Let $G$ be a \cut\ group and $P \in \Syl_3(G)$. Then $P$ is also \cut, provided one of the following holds:
	\begin{enumerate}
		\item $P$ is abelian,
		\item $P$ is  a normal subgroup of $G$.
		\end{enumerate}	
\end{proposition}

\begin{remark} Since the examples of Isaacs and Navarro of rational groups with non-rational Sylow 
$2$-subgroups are of order $2^9\cdot 3$, one  might be tempted to think that \cut\ groups of 
order $2\cdot 3^a$ could serve to find examples of \cut\ groups whose Sylow $3$-subgroups are 
not \cut. However this is not the case. Indeed, using Cayley's {normal} $2$-complement theorem together 
with the previous proposition we obtain that if $G$ is a \cut\ group with order only divisible by the 
primes $2$ and $3$ and 
with cyclic Sylow $2$-subgroups (for example $|G|=2\cdot 3^a$),  then any Sylow $3$-subgroup of $G$  is \cut.
\end{remark}

\begin{theorem}\label{G_cut_Sylow_Cut} Let $G$ be a \cut\ group and $P \in \Syl_3(G)$. Then $P$ is also \cut, provided one of the following holds:
\begin{enumerate}
 \item $G$ is supersolvable,
 \item $G$ is a Frobenius group,
 \item\label{item:simple_3} $G$ is simple,
 \item\label{item:odd_order} $G$ is of odd order and $\O_3(G)$ is abelian.
\end{enumerate}
\end{theorem}

\begin{proof} \begin{enumerate} \item Supersolvable groups have a Sylow tower \cite[VI, Satz~9.1]{Hup67}, so in particular $p$-length at most $1$ for all primes $p$. The claim thus follows from Proposition~\ref{p-length_1}.
 \item Assume that $G$ is a Frobenius group with Frobenius kernel $F$ and Frobenius complement $K$. Since $F$ and $K$ have coprime order, a Sylow $3$-subgroup $P$ of $G$ is isomorphic to a Sylow $3$-subgroup of $F$ or to a Sylow $3$-subgroup of $K$. Assume that $3$ divides $|F|$. Since $F$ is nilpotent and characteristic in $G$, $P$ is normal in $G$. Hence it is \cut\ by Proposition~\ref{prop:G_cut_Sylow_Cut_for_abelian_or_normal_P}. If $3$ divides $|K|$, then $P$ is cyclic by \cite[Theorem~11.4.5(4)]{JdR15} and it follows from Lemma~\ref{prop:centers_of_sylow_subgroups} that it is also \cut. (Of course, we could also have used the description of Frobenius \cut\ groups in \cite[Theorem~1.3]{Bac17}.)
 \item In Theorem~\ref{prop:simple_cut_groups} the simple \cut\ groups were determined. They are given along with the nilpotency class of their Sylow $3$-subgroup $P$ in Table~\ref{simple_cut_groups}, showing that there nilpotency class can exceed $2$. Using, for example \textsf{GAP}, one can check that these Sylow $3$-subgroups are again \cut.
\begin{table}[ht!]\caption{Simple \cut\ groups $G$ with the nilpotency class of their Sylow $3$-subgroup $P$}\label{simple_cut_groups}
 \begin{center}
\begin{tabular}{c|ccccccccc}\hline\hline
 $G$ & $C_2$ & $C_3$ & $A_7$ &  $A_8$ &  $A_9$ & $A_{12}$ & $L_2(7)$ & $U_3(3)$ & $U_3(5)$ \\
 $\operatorname{cl}(P)$ & $0$ & $1$ & $1$ & $1$ & $3$ & $3$ & $1$ & $2$ & $1$\\ \hline\hline
\end{tabular}

\medskip
\begin{tabular}{c|ccccccccc}\hline\hline
  $G$ & $U_4(3)$ & $U_5(2)$ & $U_6(2)$ & $S_4(3)$ & $S_6(2)$ & $O_8^+(2)$ & $ M_{11}$ & $M_{12}$ & $M_{22}$ \\
  $\operatorname{cl}(P)$ &  $3$ & $3$ & $3$ & $3$ & $3$ & $3$ & $1$ & $2$ & $1$ \\ \hline\hline
\end{tabular}

\medskip
\begin{tabular}{c|ccccccccc}\hline\hline
 $G$ & $M_{23}$ &  $M_{24}$ & $Co_1$ & $Co_2$ & $Co_3$ & $HS$ & $McL$ & $Th$ & $M$\\
   $\operatorname{cl}(P)$ &  $1$ & $2$ & $5$ & $3$ & $3$ & $1$ & $3$ & $7$ & $9$ \\ \hline\hline
\end{tabular}
 \end{center}
 \end{table}
 \item In the case a group has odd order, then inverse semi-rational is the same as semi-rational by \cite[Remark~13]{CD10}. These groups have been described in \cite[Theorem~3]{CD10}. They come in three families: $3$-groups, certain Frobenius groups and certain groups that are occasionally called $2$-Frobenius or double Frobenius groups. Only in the last case we still need to verify the claim. By \cite[Theorem~3~(2)]{CD10} they have the following structure:  $|G| = 7\cdot 3^b$ and $G$ contains a normal Frobenius subgroup of index $3$ with $\O_3(G)$ the Frobenius kernel. Moreover $\O_3(G)T \in \Syl_3(G)$ for some subgroup $T = \langle t \rangle$ of order $3$ and $G/\O_3(G)$ is the non-abelian group of order 21.
 
 Assume that $\O_3(G)$ is abelian. We aim to show that every element of the Sylow $3$-subgroup $P = \O_3(G)\langle t \rangle$ is inverse semi-rational in $P$. Firstly, we claim that in this setting $\O_3(G)$ has exponent $3$. 
Assume there is $y \in \C_{\O_3(G)}(t)$ of order $9$. 
Then $[t, y] = 1$ and hence $t \in \C_G(y)$. However, $\C_G(y) = \O_3(G)$: clearly, we have $\O_3(G) \leqslant \C_G(y)$, since the former is abelian, but also since $\O_3(G)F$ is a Frobenius group for every $F \in \Syl_7(G)$, $\C_G(y)$ cannot contain elements of order $7$ and since $y$ is inverse semi-rational in $G$, $[\N_G(\langle y \rangle) :\C_G(y)] = 3$, so indeed $\C_G(y) = \O_3(G)$. Summing up, $t \in \C_G(y) = \O_3(G)$. This yields a contradiction since $t \not\in \O_3(G)$, by the form of the Sylow $3$-subgroup. So there is no element of order $9$ in $ \C_{\O_3(G)}(t)$ and $\exp (\C_{\O_3(G)}(t)) = 3$. By \cite[Lemma~2.4]{KMS14}, $\O_3(G) = \langle\ \C_{\O_3(G)}(t)^f\ |\ f\in F\ \rangle$ for some $F \in \Syl_7(G)$, so also $\exp (\O_3(G)) = 3$, since the latter is abelian. Thus every element of $\O_3(G)$ is inverse semi-rational in $\O_3(G)$ and hence also in $P \geqslant \O_3(G)$.

Now assume that $x \in P \setminus \O_3(G)$. Then $P$ is the unique Sylow $3$-subgroup of $G$ containing $x$, because $\O_3(G) < \langle \O_3(G),\ x \rangle \leqslant P$ and $[P : \O_3(G)] = 3$. Since $x$ is inverse semi-rational in $G$, Lemma~\ref{lem:cut-3-local} asserts that $x$ is inverse semi-rational in a Sylow $3$-subgroup that contains it. Hence, $x$ is inverse semi-rational in $P$. Consequently, every element of $P$ is inverse semi-rational in $P$, i.e., $P$ is \cut.
 \qedhere
\end{enumerate}
\end{proof}

Recently, N.~Grittini proved that the assumption that $\O_3(G)$ is abelian in \eqref{item:odd_order} is superfluous: The Sylow $3$-subgroups of \cut\ groups of odd order are always \cut\ \cite{Gri20}.\\

Note that the class of groups in \eqref{item:odd_order} contains groups of $3$-length $2$. Theorem~\ref{G_cut_Sylow_Cut} could also have been proved using the dual characterization of \cut\ groups using characters. \\

 Recall that a group $G$ is called a \emph{Camina group}, if $G' \neq G$ and, for every $g \not\in G'$, the coset $gG'$ is a conjugacy class. As Camina groups are either Frobenius or $p$-groups \cite{DS96}, we immediately have the following corollary from Theorem \ref{G_cut_Sylow_Cut}:
 
 \begin{corollary}
 	A Sylow 3-subgroup of a Camina \cut\ group is again a \cut\ group.
 \end{corollary}

Using the above results and \textsf{GAP} with the \textsf{grpconst} package \cite{grpconst}, we have verified that the Sylow $3$-subgroups of all \cut\ groups of order at most $2000$, as well as for order $2^2\cdot 3^6$, $2^3\cdot 3^6$ and $2^2\cdot 3^7$, are again \cut. Naturally, the following question arises:

\begin{question}\label{que:sylwo_3_of_cut}
\textbf{Are Sylow $3$-subgroups of \cut\ groups again \cut?}
\end{question}

Question~\ref{que:sylwo_3_of_cut} essentially asks whether we can always replace the existence assertion for the Sylow $3$-subgroup in Lemma~\ref{lem:cut-3-local} by a ``for all'' statement. 

\begin{remark} Recall that (normal) subgroups of \cut\ groups need not to be \cut\ groups. But for classes for which Question~\ref{que:sylwo_3_of_cut} has a positive answer, the Sylow $3$-subgroups have this property. In view of Proposition~\ref{prop:equiv_cut}, this implies that for such groups, the centers of the Wedderburn components of the rational group algebra of \cut\ groups also influence those of its Sylow $3$-subgroups. For, if the centers of the Wedderburn components of $\QQ G$ are all rational or quadratic imaginary, then the same holds true for $\QQ P$ for $P \in \Syl_3(G)$. \end{remark}

\section{Existence of \cut\ groups}\label{sect:quantity}
In this section, we will give some indications that the class of \cut\ groups is surprisingly large in all finite groups. Recall that a $p$-group can only be a \cut\ group if $p \in \{2, 3\}$ see \cite[Theorem~1]{BMP17}. We show that in these cases, the ratio of \cut\ groups tends to one in the logarithmic sense.

\begin{proposition} \label{rem:negligible} Let $c(r)$ denote the number of \cut\ groups of order $r$ and $f(r)$ the number of all groups of order $r$. Then
\begin{equation}\label{eq:assymptotic}\tag{**} \lim_{n \to \infty} \frac{\ln c(p^n)}{\ln f(p^n)} = 1, \qquad \text{for } p \in \{2, 3\}.\end{equation} \end{proposition}

\begin{proof} In their work on asymptotic behavior of the number of $p$-groups, Higman, Sims, Newman and Seeley showed that for each prime $p$, \[p^{\frac{2}{27} m^2(m-6)}\ \leqslant\ f(p^m)\ \leqslant\ p^{\frac{2}{27}m^3 + O(m^\frac{5}{2})}, \] see \cite[Theorem~4.5, Theorem~5.7]{BNV07}. Noting that all the groups constructed by Higman to obtain the lower bound are actually \cut\ groups in the case $p = 2$ (as they are extensions of elementary abelian $2$-groups by elementary abelian $2$-groups), see \cite[Section~4]{BNV07}, shows that \eqref{eq:assymptotic} holds for $p = 2$. By \cite[Theorem~19.3]{BNV07} it follows that the number of $3$-groups of exponent $3$ (and nilpotency class $2$) has the same asymptotic behavior in the leading term, so \eqref{eq:assymptotic} also holds for $p = 3$. \end{proof}

\begin{figure}[ht!]\caption{Rational and \cut\ groups of order at most 1023}\label{f1}\centering
	\includegraphics[scale=.4]{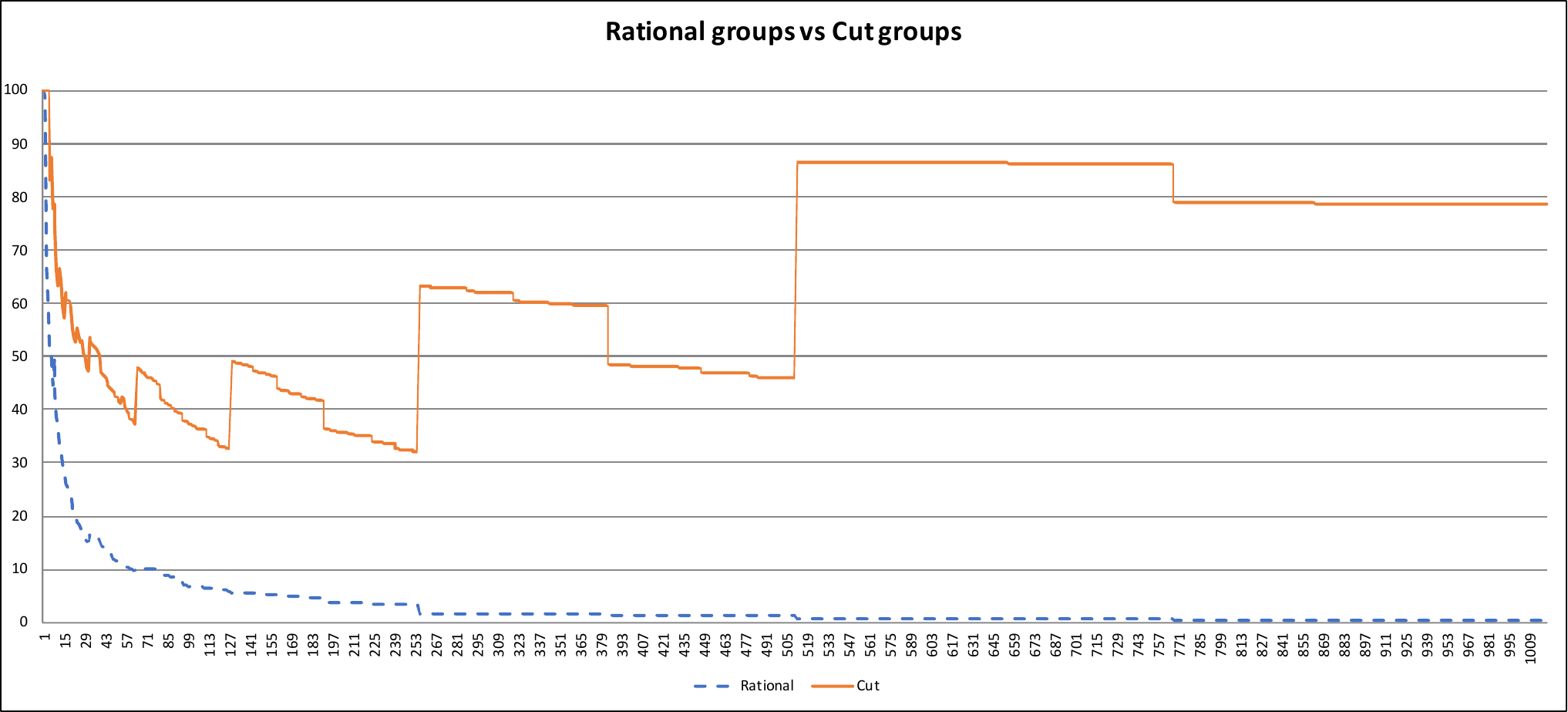}
	\renewcommand
	{\figurename}{Figure}
\end{figure}

\begin{remark} Though the above proposition shows that there are many \cut\ groups in the class of $p$-groups, it says nothing about the actual percentage. For groups of small order we now also support this with numerical data which we obtained using \textsf{GAP}  and the \textsf{SglPPow} package \cite{SglPPow}. Table~\ref{tab:2_groups} lists the number and percentage of rational as well as \cut\ $2$-groups up to order $2^9$. Observe that for the $2$-groups, the percentage of \cut\ groups increases again from order $2^8$. The number and percentage of \cut\ $3$-groups up to order $3^8$ are listed in Table~\ref{tab:3_groups} (note that the only rational group of odd order is the trivial one). Also, in Table~\ref{tab:mixed_groups}, we list the number and percentage of rational and \cut\ groups of selected mixed orders. Figure~\ref{f1} presents a chart showing the percentage of groups that are rational and \cut, up to order 1023. The horizontal and the vertical axes indicate the order and the percentage respectively. The dashed  graph gives the percentage of rational groups up to that order whereas the solid graph indicates the percentages of \cut\ groups up to that order. Note that the upward bumps for the percentage of \cut\ groups appear at $2$-powers, whereas the (visible) bumps downwards for this percentage happen at orders of the form $2^a\cdot 3$. Yet also for these orders the \cut\ groups are still surprisingly numerous, cf.\ Table~\ref{tab:mixed_groups}. Also, when considering all groups up to a certain small order, \cut\ groups are not negligible. For instance, about $86.62\%$ of the groups of order at most 512 and 78.55\% of groups of order at most 1023 are \cut\ groups, whereas  0.57\% of the groups of order at most 512 and 0.52\% of groups of order at most 1023 are rational.

\begin{question}\label{que_quant_3grps} Let $f$ and $c$ be defined as in Proposition~\ref{rem:negligible}. Is there $m \in \mathbb{Z}_{\geqslant 1}$ such that \[ \frac{c(3^{m+1})f(3^m)}{c(3^m)f(3^{m+1})} > 1, \] i.e.\ does the proportion of the \cut\ $3$-groups in all $3$-groups increase again at a certain point (as it does for $2$-groups at $2^8$)?
\end{question}

\begin{table}[ht!]\caption{Rational and \cut\ groups in small order $2$-groups}\label{tab:2_groups}
 \begin{center}
{\scriptsize
\begin{tabular}{cccc} \hline\hline
$2^n$ & \begin{minipage}{.16\textwidth} number of all groups of order $2^n$ \end{minipage} & \begin{minipage}{.16\textwidth} number of rational groups of order $2^n$ \end{minipage} & \begin{minipage}{.16\textwidth} number of \cut\ groups of order $2^n$ \end{minipage} \\
 & & \emph{percentage of rational groups of order $2^n$} & \emph{percentage of \cut\ groups of order $2^n$}  \\ \hline\hline
 $2^2 = 4$ & $2$ & $1$ & $2$ \\ 
 & & \emph{50\%} & \emph{100\%} \\ \hline
 $2^3 = 8$ & $5$ & $3$ & $4$ \\ 
 & & \emph{60\%} & \emph{80\%} \\ \hline
 $2^4 = 16$ & $14$ & $3$ & $10$ \\ 
 & & \emph{21.43\%} & \emph{71.43\%} \\ \hline
 $2^5 = 32$ & $51$ & $10$ & $33$ \\ 
 & & \emph{19.61\%} & \emph{64.71\%} \\ \hline
 $2^6 = 64$ & $267$ & $30$ & $161$ \\ 
 & & \emph{11.24\%} & \emph{60.30\%} \\ \hline
 $2^7 = 128$ & $2\ 328$ & $124$ & $1\ 349$ \\
 & & \emph{5.33\%} & \emph{57.95\%} \\ \hline
 $2^8 = 256$ & $56\ 092$ & $748$ & $37\ 593$ \\
 & & \emph{1.34\%} & \emph{67.02\%} \\ \hline
 $2^9 = 512$ & $10\ 494\ 213$ & $59\ 514$ & $9\ 127\ 858$ \\
 & & \emph{0.57\%} & \emph{86.98\%} \\ \hline
 \hline
\end{tabular}
}
\end{center}
\end{table}

\begin{table}[ht!]\caption{\textsf{Cut}\ groups in small order $3$-groups}\label{tab:3_groups}
 \begin{center}
{\scriptsize
\begin{tabular}{ccc} \hline\hline
$3^n$ & \begin{minipage}{.16\textwidth} number of all groups of order $3^n$ \end{minipage} & \begin{minipage}{.16\textwidth} number of \cut\ groups of order $3^n$ \end{minipage} \\
 & & \emph{percentage of \cut\ groups of order $3^n$}  \\ \hline\hline
 $3^2 = 9$ & $2$ & $1$ \\ 
 & & \emph{50\%}  \\ \hline
 $3^3 = 27$ & $5$ & $3$  \\ 
 & & \emph{60\%} \\ \hline
 $3^4 = 81$ & $15$ & $4$ \\ 
 & & \emph{26.67\%}  \\ \hline
 $3^5 = 243$ & $67$ & $14$  \\ 
 & & \emph{20.90\%} \\ \hline
 $3^6 = 729$ & $504$ & $96$ \\ 
 & & \emph{19.05\%}  \\ \hline
 $3^7 = 2187$ & $9\ 310$ & $595$ \\
 & & \emph{6.39\%}  \\ \hline
 $3^8 = 6561$ & $1\ 396\ 077$ & $66\ 312$ \\
 & &  \emph{5.06\%}  \\ \hline
 \hline
\end{tabular}
}
\end{center}
\end{table}

\begin{table}[ht!]\caption{Rational and \cut\ groups for some selected mixed orders}\label{tab:mixed_groups}
 \begin{center}
{\scriptsize
\begin{tabular}{cccc} \hline\hline
$m$ & \begin{minipage}{.16\textwidth} number of all groups of order $m$ \end{minipage} & \begin{minipage}{.16\textwidth} number of rational groups of order $m$ \end{minipage} & \begin{minipage}{.16\textwidth} number of \cut\ groups of order $m$ \end{minipage} \\
 & & \emph{percentage of rational groups of order $m$} & \emph{percentage of \cut\ groups of order $m$}  \\ \hline\hline
 $2^2 \cdot 5 = 20$ & $5$ & $0$ & $1$ \\ 
 & & \emph{0\%} & \emph{20\%} \\ \hline
 $2\cdot3\cdot 7 = 42$ & $6$ & $0$ & $2$ \\ 
 & & \emph{0\%} & \emph{33.33\%} \\ \hline
% $2^5\cdot 3 = 96$ & $231$ & $5$ & $78$ \\ 
% & & \emph{2.16\%} & \emph{33.77\%} \\ \hline
 $2^6\cdot 3 = 192$ & $\ 1\ 543$ & $18$ & $318$ \\ 
 & & \emph{1.17\%} & \emph{20.61\%} \\ \hline 
 $2^7\cdot 3 = 384$ & $20\ 169$ & $317$ & $2\ 279$ \\ 
 & & \emph{0.32\%} & \emph{11.30\%} \\ \hline 
 $2^3\cdot 5^2 = 400$ & $221$ & $1$ & $12$ \\ 
 & & \emph{0.45\%} & \emph{5.43\%} \\ \hline
 $2^6\cdot 3^2 = 576$ & $8\ 681$ & $45$ & $1\ 074$ \\
 & & \emph{0.52\%} & \emph{12.37\%} \\ \hline
 $2^3 \cdot 3\cdot 7^2 = 588$ & $66$ & $0$ & $4$ \\
 & & \emph{0\%} & \emph{6.06\%} \\ \hline
 $2^3\cdot 3^4 = 648$ & $757$ & $11$ & $151$ \\
 & & \emph{1.45\%} & \emph{19.95\%}  \\ \hline 
 $2^8\cdot 3 = 768$ & $1\ 090\ 235$ & $304$ & $64\ 765$ \\
 & & \emph{0.03\%} & \emph{5.94\%}  \\ \hline  
 %$2^2\cdot 3^5 = 972$ & $900$ & $8$ & $153$ \\
 %& & \emph{0.89\%} & \emph{17.00\%}  \\ \hline  
 $2\cdot 3^6 = 1458$ & $1\ 798$ & $1$ & $387$ \\
 & & \emph{0.06\%} & \emph{21.52\%}  \\ \hline  
 $2^3\cdot 3^2 \cdot 5^2 = 1800$ & $749$ & $0$ & $14$ \\
 & & \emph{0\%} & \emph{1.87\%}  \\ \hline  
 $2^3\cdot 3^5 = 1944$ & $3\ 973$ & $17$ & $525$ \\
 & & \emph{0.43\%} & \emph{13.21\%}  \\ \hline  
 \hline
\end{tabular}
}
\end{center}
\end{table}
\end{remark}

\noindent {\bf Acknowledgements.}\
We are grateful to Urban Jezernik, Gabriel Navarro, Inder Bir Singh Passi and {\'A}ngel del R{\'i}o for very interesting discussions. We want to thank J{\"u}rgen M{\"u}ller for providing us with the Sylow $3$-subgroups of some sporadic simple groups. We are grateful to Benjamin Sambale for the proof of the improved Theorem~\ref{prop:rational_classes_vs_characters}. The fourth author is thankful to Eric Jespers and FWO for the local hospitality provided for stay at Vrije Universiteit Brussel, which played a vital role in the outcome of this paper. We want to thank Adalbert and Victor Bovdi for pointing out the existence of the articles \cite{Pat78,Bov87} to us. We wish to thank the reviewers for suggesting changes that improved the readability of the paper.

% \bibliographystyle{amsalpha}
% \bibliography{Bibliography}

\newcommand{\etalchar}[1]{$^{#1}$}
\providecommand{\bysame}{\leavevmode\hbox to3em{\hrulefill}\thinspace}
\providecommand{\MR}{\relax\ifhmode\unskip\space\fi MR }
\providecommand{\MRhref}[2]{%
  \href{http://www.ams.org/mathscinet-getitem?mr=#1}{#2}
}
\providecommand{\href}[2]{#2}

\end{document}